\documentclass{amsart}

\usepackage[colorlinks=true, pdfstartview=FitV, linkcolor=blue, citecolor=blue, urlcolor=blue]{hyperref}
\usepackage{amsmath,amsfonts,amssymb,amsthm,comment}
\usepackage[all]{xy}

\newtheorem{Theorem}{Theorem}[section]
\newtheorem{Proposition}[Theorem]{Proposition} 
\newtheorem{Lemma}[Theorem]{Lemma}

\newtheorem{Corollary}[Theorem]{Corollary}

\theoremstyle{definition}

\newtheorem{Question}{Question}
\newtheorem{Comment}[Theorem]{Comment}
\newtheorem{Definition}[Theorem]{Definition}

\newcommand{\arxiv}[1]{\href{http://arxiv.org/abs/#1}{\tt arXiv:\nolinkurl{#1}}}

\newcommand{\low}{\text{low}}

\newcommand{\wt}{\operatorname{wt}}

\newcommand{\bz}{\Bbb{Z}}
\newcommand{\bc}{\Bbb{C}}

\newcommand{\g}{\mathfrak{g}}

\newcommand{\Aa}{\mathcal{A}}
\newcommand{\LL}{\mathcal{L}}

\newcommand{\Flip}{\mbox{Flip}}

\newcommand{\Id}{\mbox{Id}}

\newcommand{\C}{\mathcal{C}}

\newcommand{\barr}{\text{bar}}

\newcommand{\modqi}{{(mod \, q^{-1}\LL)}}

\theoremstyle{definition}

\begin{document}

\title[A formula for the $R$-matrix]{A formula for the $R$-matrix using a system of weight preserving endomorphisms}

\author{Peter Tingley}
\email{ptingley@math.mit.edu}
\address{Peter Tingley,
MIT dept. of math,
77 Massachusetts Ave,
Cambridge, MA, 02139}
\thanks{2000 Mathematics Subject Classification. Primary: 17B37. Secondary: 16W30.}
\thanks{This work was supported by the RTG grant DMS-0354321.}

\begin{abstract}
We give a formula for the universal $R$-matrix of the quantized universal enveloping algebra $U_q(\g).$ This is similar to a previous formula due to Kirillov-Reshetikhin and  Levendorskii-Soibelman, except that where they use the action of the braid group element $T_{w_0}$ on each representation $V$, we show that one can instead use a system of weight preserving endomorphisms. One advantage of our construction is that it is well defined for all symmetrizable Kac-Moody algebras. However we have only established that the result in equal to the universal $R$-matrix in finite type.
\end{abstract}

\maketitle 

\section{Introduction}

Let $\g$ be a finite type complex simple Lie algebra and $U_q(\g)$ the corresponding quantized universal enveloping algebra. In \cite{KR:1990} and \cite{LS}, Kirillov-Reshetikhin and Levendorskii-Soibelman developed a formula for the universal R-matrix 
\begin{equation} \label{KReq} R= (X^{-1} \otimes X^{-1}) \Delta(X), \end{equation}
where $X$ belongs to a completion of $U_q(\g)$. The element $X$ is constructed using the braid group element $T_{w_0}$ corresponding to the longest word of the braid group, and as such only makes sense when $\g$ is of finite type. 

The element $X$ in \eqref{KReq} defines a vector space endomorphism $X_V$ on each representation $V$ of $U_q(\g)$, and in fact $X$ is defined by the system of endomorphisms $\{ X_V \}$. Furthermore, any natural system of vector space endomorphisms $\{ E_V \}$ can be  represented as an element $E$ in a certain completion of $U_q(\g)$ (see \cite{Rcommutor}). The action of the coproduct $\Delta(E)$ on a tensor product $V \otimes W$ is then simply $E_{V \otimes W}$. Thus the right side of \eqref{KReq} is well defined if $X$ is replaced by $E= \{ E_V \}$.

In this note we consider the case where $\g$ is a symmetrizable Kac-Moody algebra. We define a system of weight preserving endomorphisms  $\Theta= \{ \Theta_V \}$ of all integrable highest weight representations $V$ of $U_q(\g)$. When $\g$ is of finite type, we show that
\begin{equation} \label{theta_int}
R= (\Theta^{-1} \otimes \Theta^{-1}) \Delta(\Theta),
\end{equation}
where the equality means that, for any type {\bf 1} finite dimensional modules $V$ and $W$, the actions of the two sides of \eqref{theta_int} on $V \otimes W$ agree.
We expect this remains true in other cases, although this has not been proven. 

Our endomorphisms $\Theta_V$ are not linear over the field $\bc(q)$, but are instead compatible with the automorphism which inverts $q$. For this reason, $\Theta$ cannot be realized using an element in a completion of $U_q(\g)$, and it is crucial to work with systems of endomorphisms. There is a further technically in that $\Theta_V$ actually depends on a choice of global basis for $V$. Nonetheless, we give a precise meaning to \eqref{theta_int}.

This note is organized as follows. In Section \ref{notation} we fix notation and conventions. In Section \ref{Rstuff} we review the universal $R$-matrix.  In Section \ref{autint} we review a method developed by Henriques and Kamnitzer \cite{cactus} to construct isomorphisms $V \otimes W \rightarrow W \otimes V$. In Section \ref{crystal} we state some background results on crystal bases and global bases. In Section \ref{maketheta} we construct our endomorphism $\Theta$. In Section \ref{theproof} we prove our main theorem (Theorem \ref{main}), which establishes \eqref{theta_int} when $\g$ is of finite type. In Section \ref{questions} we briefly discuss future directions for this work. 

\subsection{Acknowledgements}
We thank Joel Kamnitzer, Noah Snyder, and  Nicolai Reshetikhin for many helpful discussions.

\section{Conventions} \label{notation}
We must first fix some notation. For the most part we follow \cite{CP}.

$\bullet$ $\g$ is a symmetrizable Kac-Moody algebra with Cartan matrix $A = (a_{ij})_{i,j \in I}$ and Cartan subalgebra $\mathfrak{h}$.

$\bullet$ $ \langle \cdot , \cdot \rangle $ denotes the paring between $ \mathfrak{h} $ and $ \mathfrak{h}^\star $ and $ ( \cdot , \cdot) $ denotes the usual symmetric bilinear form on either $ \mathfrak{h}$ or $ \mathfrak{h}^\star $.  Fix the usual elements $ \alpha_i \in \mathfrak{h}^\star $ and $ H_i \in \mathfrak{h}$, and recall that $ \langle H_i, \alpha_j \rangle = a_{ij} $.  

$\bullet$ $ d_i = (\alpha_i, \alpha_i)/2 $, so that $ (H_i, H_j) = d_j^{-1} a_{ij} $ and, for all $\lambda \in \mathfrak{h}^*$, $(\alpha_i, \lambda)= d_i \langle H_i, \lambda \rangle.$  

$\bullet$ $ B $ is the symmetric matrix $ (d_j^{-1} a_{ij}) $. 

$\bullet$ $\rho \in \mathfrak{h}^*$ satisfies $\langle H_i, \rho \rangle  = 1$ for all $i$. Note that this implies $(\alpha_i, \rho) = d_i$. If $A$ is not invertible this condition does not uniquely determine $\rho$, and we simply choose any one solution. 

$\bullet$ $ H_\rho $ is the element of $ \mathfrak{h} $ such that, for any $\lambda \in \mathfrak{h}^*$, $\langle H_\rho, \lambda \rangle = (\rho, \lambda)$. In particular, $\langle H_\rho, \alpha_i \rangle=d_i$ for all $ i $.

$\bullet$ $U_q(\g)$ is the quantized universal enveloping algebra associated to $\g$, generated over $\mathbb{C}(q)$ by $E_i$, $F_i$ for all  $i \in I$, and $K_H$ for $H$ in the coweight lattice of $\g$. As usual, let $K_i= K_{d_i H_i}.$ For convenience, we recall the exact formula for the coproduct:
\begin{equation} \label{coproduct}
\begin{cases}
\Delta{E_i} & = E_i \otimes K_i + 1 \otimes E_i \\
\Delta{F_i} &= F_i \otimes 1 + K_i^{-1} \otimes F_i \\
\Delta{K_H} &= K_H \otimes K_H
\end{cases}
\end{equation}
and the following commutation relations
\begin{equation} \label{KpastEF}
K_H E_i K_H^{-1} = q^{\langle H, \alpha_i \rangle}E_i \quad \text{ and } \quad K_H F_i K_H^{-1} = q^{-\langle H, \alpha_i \rangle} F_i.
\end{equation}
At times it will be necessary to adjoin a fixed $k$-th root of $q$ to the base field $\bc(q),$ where $k$ is twice the dual Coxeter number of $\g$. 

$\bullet$ $[n]= \frac{q^n - q^{-n}}{q-q^{-1}},$ and $X^{(n)} = \frac{X^n}{[n][n-1] \cdots [2]}.$

$\bullet$ Fix a representation $V$ of $U_q(\g)$ and $\lambda \in \mathfrak{h}^*$. We say $v \in V$ is a weight vector of weight $\lambda$ if, for all $H \in \mathfrak{h}$, $K_H(v)= q^{\langle H, \lambda \rangle} v$. 

$\bullet$ $\lambda \in \mathfrak{h}^*$ is called a dominant integral weight if $\langle H_i, \lambda \rangle \in \bz_{\geq0}$ for all $i$.

$\bullet$ For each dominant integral weight $\lambda$, $V_\lambda$ is the type {\bf 1} irreducible integrable representation of $U_q(\g)$ with highest weight $\lambda$. 

$\bullet$ $B_\lambda$ is a fixed global basis for $V_\lambda$, in the sense of Kashiwara (see \cite{K}). $b_\lambda$ and $b_\lambda^{\low}$ are the highest weight and lowest weight elements of $B_\lambda$ respectively.

\section{The R-matrix} \label{Rstuff}

We briefly recall the definition of a universal $R$-matrix, and the related notion of a braiding.

\begin{Definition} 
A braided monoidal category is a monoidal category $\C$, along with a natural system of isomorphisms $\sigma^{br}_{V,W}: V \otimes W \rightarrow W \otimes V$ for each pair $V,W \in \C$, such that, for any $U,V, W \in \C$, the following two equalities hold:
\begin{equation} \label{qt} \begin{aligned} (\sigma^{br}_{U,W} \otimes \Id) \circ (\Id \otimes \sigma^{br}_{V,W}) = \sigma^{br}_{U \otimes V, W} \\
(\Id \otimes \sigma^{br}_{U,W}) \circ (\sigma^{br}_{U,V} \otimes  \Id) = \sigma^{br}_{U,V\otimes W}. 
\end{aligned}
\end{equation}
 The system $\sigma^{br} : = \{ \sigma^{br}_{V,W} \}$ is called a braiding on $\C$.
\end{Definition}

Let $\widetilde{U_q(\g) \otimes U_q(\g)}$ be the completion of $U_q(\g) \otimes U_q(\g)$ in the weak topology defined by all matrix elements of representations $V_\lambda \otimes V_\mu$, for all ordered pairs of dominant integral weights $(\lambda, \mu)$.

\begin{Definition} \label{rR}
A universal $R$-matrix is an element $R$ of $\widetilde{U_q(\g) \otimes U_q(\g)}$ such that $\sigma^{br}_{V,W} := \mbox{Flip} \circ R$ is a braiding on the category of $ U_q(\g) $ representations. 
\end{Definition}

Note in particular that, since the braiding is an isomorphism, $R$ must be invertible. It is central to the theory of quantized universal enveloping algebras that, for any symmetrizable Kac-Moody algebra $\g$, $U_q(\g)$ has a universal $R$-matrix. The universal $R$-matrix is not truly unique, but there is a well-studied standard choice. See \cite{CP} for a thorough discussion when $\g$ is of finite type, and \cite{Lusztig:1993} for the general case. 

When $\g$ is of finite type, the $R$-matrix can be described explicitly as follows. Note that the expression below is presented in the $h$-adic completion of $U_h(\g)$, whereas here we are working in $U_q(\g)$. However, it is straightforward to check that this gives a well defined endomorphism of $V \otimes W$ for any integrable highest weight $U_q(\g)$-representations $V$ and $W$, with the only difficulty being that certain fractional powers of $q$ can appear. 

\begin{Theorem} (see \cite[Theorem 8.3.9]{CP}) \label{Rnormal} Assume $\g$ is of finite type. Then the standard universal $R$ matrix for $U_q(\g)$ is given by the expression 
\begin{equation} \label{R12}
R_h = \exp \left( h \sum_{i,j} (B^{-1} )_{ij} H_i \otimes H_j \right) \prod_\beta \exp_{q_\beta} \left[ ( 1 - q_\beta^{-2} )E_\beta \otimes F_\beta \right],
\end{equation}
where the product is over all the positive roots of $\g$, and the order of the terms is such that $\beta_r$ appears to the left of $\beta_s$ if $r > s$. \qed
\end{Theorem}

We will not explain all the notation in \eqref{R12}, since the only thing we use is the fact that $E_\beta$ acts as $0$ on any highest weight vector, and so the product in the expression acts as the identity on $b_\lambda \otimes c \in V_\lambda \otimes V_\mu$.

\section{Constructing isomorphisms using systems of endomorphisms} \label{autint}

Here and throughout this note a representation of $U_q(\g)$ will mean a direct sum of possibly infinitely many of the irreducible integrable type {\bf 1} representations $V_\lambda$. We note that the category of such representations is closed under tensor product. When $\g$ is of finite type, we can restrict to finite direct sums, or equivalently finite dimensional type {\bf 1} modules, since this category is already closed under tensor product. 

In this section we review a method for constructing natural systems of isomorphisms $\sigma_{V,W}: V \otimes W \rightarrow W \otimes V$. This idea was used by Henriques and Kamnitzer in \cite{cactus}, and was further developed in \cite{Rcommutor}. The data needed to construct such a system is:

\begin{enumerate}

\item An algebra automorphism $C_\xi$ of $U_q(\g)$ which is also a coalgebra anti-automorphism.

\item \label{diagg} A natural system of invertible vector space endomorphisms $\xi_V$ of each representation $V$ of $U_q(\g)$ which is compatible with $C_\xi$ in the sense that the following diagram commutes for all $V$:
\begin{equation*}
\xymatrix{
V  \ar@(dl,dr) \ar@/ /[rrr]^{\xi_V} &&& V \ar@(dl,dr) \\
U_q(\g)  \ar@/ /[rrr]^{\C_{\xi}} &&&   U_q(\g). \\
}
\end{equation*}
\end{enumerate}
It follows immediately from the definition of coalgebra anti-automorphism that 
\begin{equation} \label{seq1} \sigma^\xi_{V,W} := \Flip \circ (\xi_V^{-1} \otimes \xi_W^{-1}) \circ \xi_{V \otimes W}
\end{equation}
is an isomorphism of $U_q(\g)$ representations from $V \otimes W$ to $W \otimes V$, where $\Flip$ is the map from $V \otimes W$ to $W \otimes V$ defined by $\Flip (v \otimes w)= w \otimes v.$ 

We will normally denote the system $\{ \xi_V \}$ simply by $\xi$, and will denote the action of $\xi$ on the tensor product of two representations by $\Delta(\xi)$. This is justified since, as explained in \cite{Rcommutor}, $\xi$ in fact belongs to a completion of $U_q(\g)$, and the action of $\xi$ on $V \otimes W$ is calculated using the coproduct. With this notation $\sigma^\xi:= \{ \sigma^\xi_{V,W} \}$ can be expressed as
\begin{equation} \label{seq2} \sigma^\xi = \Flip \circ (\xi^{-1} \otimes \xi^{-1}) \circ \Delta(\xi).
\end{equation}

In the current work we require a little more freedom: we will sometimes use automorphisms $C_\xi$ of $U_q(\g)$ which are not linear over $\bc(q)$, but instead are bar-linear (i.e. invert $q$). This causes some technical difficulties, which we deal with in Section \ref{maketheta}. Once we make this precise, we will use all the same notation for a bar-linear $C_\xi$ and compatible system of $\bc$ vector space automorphisms $\xi$ as we do in the linear case, including using $\Delta(\xi)$ to denote $\xi$ acting on a tensor product.

\begin{Comment} Since the representations we are considering are all completely reducible, to describe the data $( \C_\xi, \xi )$ it is sufficient to describe $C_\xi$ and to give the action of $\xi_{V_\lambda}$ on any one vector $v$ in each irreducible representation $V_\lambda$. This is usually more convenient then describing $\xi_{V_\lambda}$ explicitly. Of course, the choice of $C_\xi$ imposes a restriction on $\xi_{V_\lambda} (v)$, so when we give such a description of $\xi$, we must check that the action on our chosen vector in each $V_\lambda$ is compatible with $C_\xi$.
\end{Comment}

\begin{Comment} \label{when-aut}
If $C_\xi$ is an coalgebra automorphism as opposed to a coalgebra anti-automorphism, the same arguments show that $(\xi_V^{-1} \otimes \xi_W^{-1}) \circ \xi_{V \otimes W}: V \otimes W \rightarrow V \otimes W$ is an isomorphism. 
\end{Comment}

\section{Crystal bases and Global bases} \label{crystal}

In order to extend the construction described in the Section \ref{autint} to include bar linear $\xi$, we will need to use some results concerning crystal bases and global bases. We state only what is relevant to us, and refer the reader to \cite{K} for a more complete exposition. Unfortunately, the conventions in \cite{K} and \cite{CP} do not quite agree. In particular, the theorems from \cite{K} that we will need are stated in terms of a different coproduct, so we have modified them to match our conventions.

\begin{Definition} \label{Kashiwara_operators} Fix an integrable highest weight representation $V$ of $U_q(\g)$.  Define the Kashiwara operators $ \tilde{F}_i, \tilde{E}_i : V \rightarrow V $ by linearly extending
\begin{equation}
\begin{cases}
\tilde{F}_i(F_i^{(n)}(v)) = F_i^{(n+1)} (v) \\
\tilde{E}_i (F_i^{(n)}(v))= F_i^{(n-1)} (v).
\end{cases}
\end{equation}
for all $ v \in V $ such that  $E_i (v)=0$. 
\end{Definition}

\begin{Definition}
Let $\Aa_\infty = \mathbb{C}[q^{-1}]_0$ be the algebra of rational functions in $q^{-1}$ over $\bc$ whose denominators are not divisible by $q^{-1}$.
\end{Definition}

\begin{Definition} \label{crystaldef}
A crystal basis of a representation $V$ (at $q=\infty$) is a pair $(\LL, \widetilde B) $, where $ \LL $ is an $\Aa_\infty$-lattice  of $ V$ and $\widetilde B$ is a basis for $ \LL/{q^{-1}}\LL$, such that
\begin{enumerate}
\item $\LL $ and $ \widetilde B $ are compatible with the weight decomposition of $V $.
\item $\LL $ is invariant under the Kashiwara operators and $ \widetilde B \cup 0 $ is invariant under their residues $ e_i := \tilde{E}_i^\modqi, f_i := \tilde{F}_i^\modqi : \LL/{q^{-1}}\LL \rightarrow \LL/{q^{-1}} \LL$.
\item For any $ b, b' \in \widetilde B $, we have $ e_i b = b' $ if and only if $ f_i b' = b $.
\end{enumerate}
\end{Definition}

\begin{Definition}
Let $(\LL, \widetilde B)$ be a crystal basis for $V$. The highest weight elements of $\widetilde B$ are those $b \in \widetilde B$ such that, for all $i$, $e_i(b) =0$.
\end{Definition}

\begin{Proposition} (see \cite{K}) 
Each $V_\lambda$ has a crystal basis $(\LL_\lambda, \widetilde B_\lambda)$. Furthermore, $(\LL_\lambda, \widetilde B_\lambda)$ has a unique highest weight element, and this occurs in the $\lambda$ weight space. \qed
\end{Proposition}

\begin{Theorem} \cite[Thoerem 1]{K} \label{crystal_tensor}
Let $ V, W $ be representations with crystal bases $(\LL, \widetilde A)$ and $(\mathcal{M}, \widetilde B) $ respectively.  Then $( \LL \otimes \mathcal{M}, \widetilde A \otimes \widetilde B) $ is  a crystal basis of $ V \otimes W $. Furthermore, the highest weight elements of $\widetilde A \otimes \widetilde B$ are all of the form $a^{high} \otimes b$, where $a^{high}$ is a highest weight element of $\widetilde A$. \qed
\end{Theorem}

\begin{Definition} \label{Slmn}
Let $(\LL_\lambda, \widetilde B_\lambda)$ and $(\LL_\mu, \widetilde B_\mu)$ be crystal bases for $V_\lambda$ and $V_\mu.$ Set
\begin{equation*} {S}_{\lambda, \mu}^\nu:= \{ b \in \widetilde B_\mu : b_\lambda \otimes b \text{ is a highest weight element of } \widetilde B_\lambda \otimes \widetilde B_\mu \text{ of weight } \nu \}. \end{equation*}
\end{Definition}

For any $V_\lambda$, and any choice of highest weight vector $b_\lambda \in V_\lambda$, there is a canonical choice of basis $B_\lambda$  for $V_\lambda$, which contains $b_\lambda$, and such that $(B_\lambda+ q \LL, \LL)$ is a crystal basis for $V$, where $\LL$ is the $\Aa_\infty$-span of $B_\lambda$. That is not to say there is a unique basis for $V_\lambda$ satisfying these two conditions, only that one can find a canonical ``good" choice. This is known as the global basis for $V_\lambda$. A complete construction can be found in \cite{K}, although here we more closely follow the presentation from \cite[Chapter 14.1C]{CP}. In the present work we simply use the fact that the global basis exists, and state the properties of $B_\lambda$ that we need. 

\begin{Definition} \label{bardef}
$C_\barr: U_q(\g) \rightarrow U_q(\g)$ is the $\bc$-algebra involution defined by
\begin{equation}
\begin{cases}
C_\barr(E_i) = E_i \\
C_\barr(F_i) = F_i \\
C_\barr(K_i) = K_i^{-1}  \\
C_\barr(q) = q^{-1}.
\end{cases}
\end{equation}
\end{Definition}

\begin{Theorem} (Kashiwara \cite{K}) \label{canonical_basis}
Fix a highest weight vector $b_\lambda \in V_\lambda$. There is a canonical choice of a ``global" basis $B_\lambda$ of $V_\lambda$. This has the properties (although is not defined by these alone) that:
\begin{enumerate}
\item $b_\lambda \in B_\lambda$.

\item $B_\lambda$ is a weight basis for $V_\lambda$.

\item Let $\LL$ be the $\Aa_\infty$ span of $B_\lambda$. Then $(B_\lambda + q^{-1} \LL, \LL)$ is a crystal basis for $V_\lambda$.

\item \label{gbar} Define the involution $\barr_{(V_\lambda, B_\lambda)}$ of $V_\lambda$ by $\barr_{(V_\lambda, B_\lambda)}(f(q) b) = f(q^{-1}) b$ for all $f(q) \in \bc(q)$ and $b \in B_\lambda$. Then  $\barr_{(V_\lambda, B_\lambda)}$ is compatible with $C_\barr$, in the sense discussed in Section \ref{autint}. 
\end{enumerate} 
Furthermore, if a different highest weight vector is chosen, $B_\lambda$ is multiplied by an overall scalar. \qed
\end{Theorem}

\begin{Definition}
If $V$ is any (possibly reducible) representation of $U_q(\g)$, we say a basis $B$ of $V$ is a global basis if there is a decomposition of $V$ into irreducible components such that $B$ is a union of global bases for the irreducible pieces.
\end{Definition}

\section{The system of endomorphisms $\Theta$} \label{maketheta}

We now introduce a $\bc$-algebra automorphism $C_\Theta$ of $U_q(\g)$. Notice that this inverts $q$, so it is not a $\bc(q)$ algebra automorphism, but is instead $\barr$ linear:
\begin{equation}
\begin{cases}
C_\Theta (E_i) =   E_{i} K_i^{-1} \\
C_\Theta (F_i) =   K_i F_{i} \\
C_\Theta (K_i) = K_{i}^{-1} \\
C_\Theta (q) = q^{-1}.
\end{cases}
\end{equation}
One can check that $C_\Theta$ is a well defined algebra involution and a coalgebra anti-involution. In order to use the methods of section \ref{autint}, we must define a $\bc$-vector space automorphism $\Theta_{V_\lambda}$ of each $V_\lambda$ which is compatible with $C_\Theta$. This is complicated by the fact that $C_\Theta$ does not preserve the $\bc(q)$ algebra structure, but instead inverts $q$. We must actually work in the category of representations with chosen global bases. An element of this category will be denoted $(V,B)$, where $B$ is the chosen global basis of $V$. 

\begin{Definition} \label{Thetadef} Fix a global basis $B_\lambda$ for $V_\lambda$. 
The action of $\Theta_{(V_\lambda, B_\lambda)}$ on $V_\lambda$ is defined by requiring that it be compatible with $C_\Theta$, and that $\Theta_{(V_\lambda, B_\lambda)}(b_\lambda)=  q^{-( \lambda, \lambda)/2 + ( \lambda, \rho )} b_\lambda$. This is extended by naturality to define $\Theta_{(V,B)}$ for any (possibly reducible) $V$. 
\end{Definition}

\begin{Comment}
To ensure that Definition \ref{Thetadef} makes sense, one must check that there is a map which sends $b_\lambda$ to $q^{-( \lambda, \lambda)/2 + ( \lambda, \rho )} b_\lambda$ and is compatible with $C_\Theta$. This amounts to checking that $b_\lambda$ is still a highest weight vector if the action of $U_q(\g)$ is twisted by the automorphism $C_\Theta$, and is not difficult.  
\end{Comment}

\begin{Comment}
In some cases $\Theta$ acts on a weight vector as multiplication by a fractional power of $q$. To be completely precise we should adjoin a fixed $k^{th}$ root of unity to the base field $\bc(q)$, where $k$ is twice the dual Coxeter number of $\g$. This causes no significant difficulties.
\end{Comment}

The construction described in Section \ref{autint} uses the action of $\xi_{V \otimes W}$ on $V \otimes W$. Thus we will need to define how $\Theta$ acts on a tensor product. In particular, we need a well defined notion of tensor product in the category of representations with chosen global bases. 

\begin{Definition} \label{pieces}
Let $V_{\lambda, \mu}^\nu$ denote the isotypic component of $V_\lambda \otimes V_\mu$ with highest weight $\nu$. 
Let $\displaystyle V_{\lambda, \mu}^{> \nu}:= \bigcup_{\gamma > \nu} V_{\lambda, \mu}^\gamma, \quad V_{\lambda, \mu}^{\geq \nu}:= \bigcup_{\gamma \geq \nu} V_{\lambda, \mu}^\gamma, \quad \text{ and } \quad Q_{\lambda, \mu}^\nu := V_{\lambda, \mu}^{\geq \nu} \huge{/} V_{\lambda, \mu}^{> \nu}.$ 
Here we use the partial order of the weight lattice where $\gamma \geq \nu$ iff $\gamma-\nu$ is a non-negative linear combination of the $\alpha_i$. 
\end{Definition}

\begin{Comment} \label{incl-iso} It is clear that the inclusion $V_{\lambda, \mu}^\nu \hookrightarrow V_{\lambda, \mu}^{\geq \nu}$ descends to an isomorphism from $V_{\lambda, \mu}^\nu$ to $Q_{\lambda, \mu}^\nu$.
\end{Comment}

\begin{Definition} \label{tensordef}
The tensor product $(V_\lambda, B_\lambda) \otimes (V_\mu, B_\mu)$ is defined to be  $(V_\lambda \otimes V_\mu, A)$, where $A$ is the unique global basis of $V \otimes W$ such that the projections of the highest weight elements of $A$ of weight $\nu$ in $Q_{\lambda, \mu}^\nu$ are equal to the projections of $b_\lambda \otimes b$ for those $b \in  S^\nu_{\lambda, \mu}$.  This is well defined by Comment \ref{incl-iso}. Extend by naturality to can a tensor product $(V, B) \otimes (W, C)$ for possibly reducible $V$ and $W$. 
\end{Definition}

\section{Proof that we obtain the $R$-matrix when $\g$ is of finite type} \label{theproof}

The proof of our main theorem uses a relationship between the $R$-matrix and the braid group element $T_{w_0}$ first observed in \cite{KR:1990} and \cite{LS}. Thus for this section we must restrict to finite type. We hope the result will prove to be true in greater generality, but establishing this would certainly require a different approach. We start by introducing a few more automorphisms of $U_q(\g)$ and of its representations.

\begin{Definition}
Let $\theta$ to be the diagram automorphism such that $w_0 (\alpha_i) = - \alpha_{\theta(i)},$ where $w_0$ is the longest element in the Weyl group.
\end{Definition}

\begin{Definition} \label{defGamma}
$C_\Gamma$ is the $\bc$-Hopf algebra automorphism of $U_q(\g)$ defined by
\begin{equation}
\begin{cases}
C_{\Gamma}(E_i) = - K_{\theta(i)} F_{\theta(i)}\\
C_{\Gamma}(F_i) = - E_{\theta(i)}K_{\theta(i)}^{-1}  \\
C_{\Gamma}(K_i)= K_{\theta(i)} \\
C_\Gamma (q) = q^{-1}.
\end{cases}
\end{equation}
Define the action of $\Gamma_{(V_\lambda, B_\lambda)}$ on $V_\lambda$ to be the unique $\bc$-linear endomorphism of each $V_\lambda$ which is compatible with $C_\Gamma$, and which is normalized so that $\Gamma (b_\lambda) = b_\lambda^\low$. Extend this by naturality to get the action of $\Gamma_{(V,B)}$ on any (possible reducible) representation $V$ with chosen global basis $B$. \end{Definition}

\begin{Comment}
It is a simple exercise to check that $C_\Gamma$ is in fact a Hopf algebra automorphism, and is compatible with a $\bc$-vector space automorphism of $V_\lambda$ which takes $b_\lambda$ to $b_\lambda^\low$.
\end{Comment}

\begin{Definition} \label{defGamma} \label{defJ} $C_{T_{w_0}}$ and $C_J$  are the $\bc(q)$-algebra automorphisms of $U_q(\g)$ defined by
\begin{align}
& \begin{cases}
C_{T_{w_0}}(E_i) = -F_{\theta(i)} K_{\theta(i)} \\
C_{T_{w_0}}(F_i) = -K_{\theta(i)}^{-1} E_{\theta(i)} \\
C_{T_{w_0}}(K_H) = K_{w_0(H)}, \text{ so that } C_{T_{w_0}}(K_i) = K_{\theta(i)}^{-1},
\end{cases} \\
& \begin{cases}
C_{J}(E_i) = K_i E_i  \\
C_{J}(F_i) = F_i K_i^{-1}\\
C_{J}(K_H) =  K_H.
\end{cases}
\end{align}
The systems of $\bc(q)$-vector space automorphisms $T_{w_0}$ and $J$ of each $V_\lambda$ are the unique automorphisms which are compatible with $C_{T_{w_0}}$ and $C_J$ respectively, and such that $T_{w_0} (b_\lambda^\text{low})=b_\lambda$ and $J(b_\lambda)= q^{( \lambda, \lambda) /2 + ( \lambda, \rho )} b_\lambda$, where $b_\lambda$ and $b_\lambda^{\text{low}}$ are the highest and lowest weight elements in some global basis $B_\lambda$.
\end{Definition}

\begin{Comment}
It is  straight forward exercise to show that the formulas in Definition \ref{defGamma} do define algebra automorphisms of $U_q(\g)$ and compatible vector space automorphisms of each $V_\lambda$. There is an action of the braid group on each $V_\lambda$, and $T_{w_0}$ is in fact the action of the longest element (for an appropriate choice of conventions). Note also that $J$ and $T_{w_0}$ do not depend on the choice of global basis as they are stable under simultaneously rescaling $b_\lambda$ and $b_\lambda^{\text{low}}.$ All of this is discussed in \cite{Rcommutor}.
\end{Comment}

\begin{Lemma} \label{whatisgamma} The following identities hold:
\begin{enumerate}
\item \label{Gebt} $\Gamma_{(V,B)} =  \barr_{(V,B)} \circ T_{w_0}^{-1},$ 

\item \label{TeK} $\Theta_{(V, B)}= K_{2H_\rho} \circ \barr_{(V, B)} 
\circ J,$

\item  \label{Jpower} For any weight vector $v \in V$ with $\wt(v)=\mu$, $J(v) =  q^{( \mu, \mu) /2 + ( \mu, \rho )} v,$

\item \label{Tpower} For any $b \in B$ with $\wt(b)=\mu$, $\Theta_{(V, B)} (b) =  q^{-( \mu, \mu)/2 + ( \mu, \rho )} b,$

\item \label{GT=JT} $ \Gamma^{-1}_{(V,B)} \circ \Theta_{(V,B)} =J T_{w_0}.$ 

\end{enumerate}
Here $\barr_{(V,B)}$ is the involution defined in Theorem \ref{canonical_basis}, part \eqref{gbar}. 
\end{Lemma}

\begin{proof} Let $C_{K_{2H_\rho}}$ be the algebra automorphism of $U_q(\g)$ defined by $C_{K_{2H_\rho}}(X) =  K_{2H_\rho} X K_{2H_\rho}^{-1}.$ It follows directly from \eqref{KpastEF} that
\begin{equation} \label{C2r}
C_{K_{2H_\rho}} (K_i^{-1}E_i)= E_iK_i^{-1} \quad \text{ and } C_{K_{2H_\rho}} (F_i K_i)= K_i F_i.
\end{equation}
Using \eqref{C2r} and the relevant definitions, a simple check on generators shows that 
\begin{align}
C_{\Gamma} =  C_{\barr} \circ C_{T_{w_0}}^{-1}, \quad  \hspace{-0.2cm} C_{\Theta}= C_{K_{2H_\rho}} \circ C_{\barr} \circ C_{J}, 
\quad   \hspace{-0.2cm} \text{ and } \quad \hspace{-0.2cm} C_{\Gamma}^{-1} \circ C_{\Theta} =C_{J} \circ C_{T_{w_0}}.
\end{align}
Thus, to prove \eqref{Gebt}, \eqref{TeK} and \eqref{GT=JT}, it suffices to check each identity when each side acts on any one chosen vector $b$ in each $V_\lambda$. For parts \eqref{Gebt} and \eqref{TeK}, choose $b=b_\lambda$ and the identity is immediate from definitions. 

For part \eqref{Jpower}, it is sufficient to consider $V=V_\lambda$. By Definition \ref{defJ}, \eqref{Jpower} holds for $b=b_\lambda$. Furthermore, vectors of the form $F_{i_k} \cdots F_{i_1} b_\lambda$ generate $V_\lambda$ as a $\bc(q)$ module. Assume that $v$ is a weight vector of weight $\mu$, and $J(v) = q^{( \mu, \mu) /2 + ( \mu, \rho )} v$. Fix $i \in I$. Then
\begin{equation}
\begin{aligned}
J(F_i v) & = C_J(F_i) J(v)
=F_i K_i^{-1} q^{( \mu, \mu) /2 + ( \mu, \rho )} v
= F_i q^{-\langle d_i H_i, \mu \rangle}  q^{( \mu, \mu) /2 + ( \mu, \rho )} v \\
& =q^{-(\alpha_i, \mu)}  q^{( \mu, \mu) /2 + ( \mu, \rho )} v
=q^{(\mu-\alpha_i, \mu-\alpha_i)/2+ (\mu-\alpha_i, \rho)} v.
\end{aligned}
\end{equation}
The claim now follows by induction on $k$. 

Part \eqref{Tpower} follows by directly calculating the action of the right side of \eqref{TeK} on $b$ and using Part \eqref{Jpower} to evaluation the action of $J$.

The definitions of $\Theta_{(V,B)}$ and $\Gamma_{(V,B)}$, along with parts \eqref{Jpower} and \eqref{Tpower}, now immediately imply that  $\Gamma^{-1}_{(V_\lambda,B_\lambda)} \circ \Theta_{(V_\lambda,B_\lambda)}(b_\lambda^\text{low}) = J T_{w_0}(b_\lambda^{\text{low}})= q^{(\lambda, \lambda)/2+(\lambda, \rho)} b_\lambda$, completing the proof of \eqref{GT=JT}. 
\end{proof}

We also need the following construction of the $R$ matrix due to Kirillov-Reshetikhin and Levendorskii-Soibelman. Due to a different choice of conventions, our $T_{w_0}$ is $ K_{H_\rho}^{-1} T_{w_0}^{-1}$ in those papers, so we have modified the statement accordingly. 
 As with Theorem \ref{sR}, this expression is written using the $h$-adic completion of $U_h(\g)$, but gives a well defined action on $V \otimes W$ for any finite dimensional type {\bf 1} $U_q(\g)$-module. 

\begin{Theorem}{\cite[Theorem 3]{KR:1990}, \cite[Theorem 1]{LS}} \label{sR}
The standard universal $R$-matrix can be realized as
\begin{equation} \label{R21}
R=  \exp \left( h \sum_{i, j \in I} (B^{-1})_{ij} H_i \otimes H_j \right) (T_{w_0}^{-1} \otimes T_{w_0}^{-1}) \Delta(T_{w_0}). 
\end{equation} 
\qed
\end{Theorem}

\begin{Corollary} \label{xR_R_cor}
$\displaystyle (T_{w_0}^{-1} \otimes T_{w_0}^{-1}) \Delta(T_{w_0})= \prod_\beta \exp_{q_\beta} \left[ ( 1 - q_\beta^{-2} )E_\beta \otimes F_\beta \right],$

\noindent 
 where the product is over all the positive roots of $\g$, and the order of the terms is such that $\beta_r$ appears to the left of $\beta_s$ if $r > s$.
\end{Corollary}

\begin{proof}
Follows immediately from Theorems \ref{Rnormal} and \ref{sR}, since the action of $R$ on $V_\lambda \otimes V_\mu$ is invertible.
\end{proof}

As discussed in \cite{Rcommutor}, the following is equivalent to Theorem \ref{sR}:

\begin{Corollary} (see \cite[Comment 7.3]{Rcommutor}) \label{xR}
Let $X=JT_{w_0}$. Then $$R= (X^{-1} \otimes X^{-1}) \Delta(X).$$ \qed 
\end{Corollary} 

\begin{Lemma} \label{Gamma_Lemma} Fix type {\bf 1} finite dimensional $U_q(\g)$ representations with chosen global bases $(V,B)$ and $(W, C)$. The operator $(\Gamma_{(V,B)} \otimes \Gamma_{(W,C)}) \Gamma_{(V \otimes W, A)})^{-1}$  acts on $V \otimes W$ as the identity, where $A$ is the global basis of $V \otimes W$ constructed from $B$ and $C$ in Definition \ref{tensordef}. \end{Lemma}

\begin{proof} It suffices to consider the case when $V=V_\lambda$ and $W=V_\mu$ are irreducible. Set  
\begin{equation}
m^\Gamma:=(\Gamma_{(V_\lambda, B_\lambda)} \otimes \Gamma_{(V_\mu, B_\mu)})(\Gamma_{(V_\lambda \otimes V_\mu, A)})^{-1}: V_\lambda \otimes V_\mu \rightarrow V_\lambda \otimes V_\mu.
\end{equation}
We must show that $m^\Gamma$ is the identity.
$C_\Gamma$ is a Hopf algebra automorphism of $U_q(\g)$, so, as in Section \ref{autint}, it follows that 
$m^\Gamma$
is an automorphism of $U_q(\g)$ representations. In particular, $m^\Gamma$ preserves isotypic components of $V_\lambda \otimes V_\mu$ and acts on each sub-quotient $Q_{\lambda, \mu}^\nu$ (see Definition \ref{pieces}). It is sufficient to show that the action on $Q_{\lambda, \mu}^\nu$ is the identity for all $\nu$. In fact it is sufficient to consider the action on the highest weight space of $Q_{\lambda, \mu}^\nu$, since this generates $Q_{\lambda, \mu}^\nu$. This highest weight space has a basis consisting of  $\{ \overline{\overline{b_\lambda \otimes b}} : b \in {S}_{\lambda, \mu}^\nu \}$, where $S_{\lambda, \mu}^\nu$ is as in Definition \ref{Slmn} and we use the notation $\overline{\overline{a \otimes b}}$ to denote the image of $a \otimes b$ in $Q_{\lambda, \mu}^\nu$.

By Lemma \ref{whatisgamma} part \eqref{Gebt} and Corollary \ref{xR_R_cor}, 
\begin{equation} \label{Rbig}
\begin{aligned}
m^\Gamma &=  ( \barr_{(V_\lambda, B_\lambda)} \otimes \barr_{(V_\mu, B_\mu)})  (T_{w_0}^{-1} \otimes T_{w_0}^{-1}) \Delta(T_{w_0}){\barr_{(V_\lambda \otimes V_\mu, A)}} \\
&= ( \barr_{(V_\lambda, B_\lambda)} \otimes \barr_{(V_\mu, B_\mu)})    \prod_\beta \exp_{q_\beta} \left[ ( 1 - q_\beta^{-2} )E_\beta \otimes F_\beta \right]  {\barr_{(V_\lambda \otimes V_\mu, A)}},
\end{aligned}
\end{equation}
For convenience, set
\begin{equation}
\Psi:=  ( \barr_{(V_\lambda, B_\lambda)} \otimes \barr_{(V_\mu, B_\mu)})  \prod_\beta \exp_{q_\beta} \left[ ( 1 - q_\beta^{-2} )E_\beta \otimes F_\beta \right] .
\end{equation}
Both $m^\Gamma$ and $ {\barr_{(V_\lambda \otimes V_\mu, A)}}$ act in a well defined way on each $Q_{\lambda, \mu}^\nu$, which implies that $\Psi $ does as well. 

The global basis $A$ was chosen so that $ {\barr_{(V_\lambda \otimes V_\mu, A)}} (\overline{\overline{b_\lambda \otimes b}}) = \overline{\overline{b_\lambda \otimes b}}$ (see Definition \ref{tensordef}).
Since all $E_\beta$ kill $b_\lambda$ and $ ( \barr_{(V_\lambda, B_\lambda)} \otimes \barr_{(V_\mu, B_\mu)})  $ preserves $b_\lambda \otimes b$ by definition, we see that $\Psi (b_\lambda \otimes b) =b_\lambda \otimes b$, and, taking the image in $Q_{\lambda, \mu}^\nu$, $\Psi (\overline{\overline{b_\lambda \otimes b}})= \overline{\overline{b_\lambda \otimes b}}.$ Thus, using \eqref{Rbig}, we see that $m^\Gamma$ acts on $\overline{\overline{b_\lambda \otimes b}}$ as the identity. The lemma follows. \end{proof}

\begin{Theorem} \label{main}Fix type {\bf 1} finite dimensional $U_q(\g)$ representations with chosen global bases $(V,B)$ and $(W, C)$.  Then $\big(\Theta_{(V,B)}^{-1} \otimes \Theta_{(W,C)}^{-1}\big) \Theta_{(V \otimes W, A)}$ acts on $V \otimes W$ as the standard $R$-matrix, where $A$ is the global basis of $V \otimes W$ constructed from $B$ and $C$ in Definition \ref{tensordef}. This holds independently of the choices of global bases $B$ and $C$. 
\end{Theorem}

\begin{proof}
By Corollary \ref{xR} and Lemma \ref{whatisgamma} part \eqref{GT=JT} 
\begin{equation}
\begin{aligned}
R & = 
((J T_{w_0})^{-1} \otimes (J T_{w_0})^{-1}) \Delta(J T_{w_0}) \\
&=
(\Theta_{(V,B)}^{-1} \otimes \Theta_{(W,C)}^{-1}) (\Gamma_{(V,B)} \otimes \Gamma_{(W,C)}) (\Gamma_{(V \otimes W, A)})^{-1} \Theta_{(V \otimes W, A)}.
\end{aligned}
\end{equation} 
By Lemma \ref{Gamma_Lemma}, the $ (\Gamma_{(V,B)} \otimes \Gamma_{(W,C)}) (\Gamma_{(V \otimes W, A)})^{-1} $ that appears acts as the identity.
\end{proof}

\begin{Comment} \label{match-intro} By Theorem \ref{main}, the composition 
\begin{equation} \label{the-thing}
\big(\Theta_{(V,B)}^{-1} \otimes \Theta_{(W,C)}^{-1}\big) \Theta_{(V \otimes W, A)}
\end{equation}
does not depend on the choices on global bases $B$ and $C$. Introducing the notation $\Delta(\Theta)$ to mean $\Theta_{(V \otimes W, A)}$ and dropping the subscripts, we can interpret $(\Theta^{-1} \otimes \Theta^{-1}) \Delta(\Theta)$ as \eqref{the-thing} calculated using any global bases $B$ and $C$. Then Theorem \ref{main} becomes \eqref{theta_int} from the introduction. We also note that $\Theta_{(V,B)}$ is easily seen to be an involution, so the inverses in \eqref{the-thing} are perhaps unnecessary. 
\end{Comment}

\section{Future directions} \label{questions} 

Although we have only proven Theorem \ref{main} when $\g$ is of finite type, much of the construction works in greater generality. We did not assume $\g$ was finite type in Section \ref{maketheta}, so the expression $\big(\Theta_{(V,B)}^{-1} \otimes \Theta_{(W,C)}^{-1}\big) \Theta_{(V \otimes W, A)}$ makes sense for any symmetrizable Kac-Moody algebra. Since $C_\Theta$ is a coalgebra-antiautomorphism, the methods from Section \ref{autint} imply that 
\begin{equation} \label{the-sigma}
\Flip \circ \big(\Theta_{(V,B)}^{-1} \otimes \Theta_{(W,C)}^{-1}\big) \Theta_{(V \otimes W, A)}
\end{equation}
is an isomorphism of representations. Furthermore, it is true in general that \eqref{the-sigma} does not depend on the choice of $B$ and $C$. To see why, it is sufficient to consider the case when $V=V_\lambda$ and $W=V_\mu$ are irreducible. Then the global bases $B_\lambda$ and $B_\mu$ are unique up multiplication by an overall scalar. It is straightforward to see that if $B_\lambda$ (or $B_\mu$) is scaled by a constant $z$, then $A$ is scaled by $z$ as well, and from there that both $\Theta_{(V_\lambda, B_\lambda)}$ and $\Theta_{(V_\lambda \otimes V_\mu, A)}$ are scaled by $z/\bar z$, where $\bar z$ is obtained from $z$ by inverting $q$. Thus the composition is unchanged. 

As in Comment \ref{match-intro}, we can now make sense of the expression $(\Theta^{-1} \otimes \Theta^{-1}) \Delta(\Theta)$ for all symmetrizable Kac-Moody algebras $\g$. The fact that \eqref{the-sigma} defines an isomorphism is one of the properties required of a universal $R$-matrix. However, we have not proven the crucial equalities \eqref{qt}. Thus we ask: 

\begin{Question}
Is $(\Theta^{-1} \otimes \Theta^{-1}) \Delta(\Theta)$ a universal $R$-matrix for $U_q(\g)$ if $\g$ is a general symmetrizable Kac-Moody algebra?  If yes, is it the standard $R$-matrix?
\end{Question}

\end{document}